\documentclass[reqno]{amsart}
\usepackage{cite}
\usepackage[russian,english]{babel}
\usepackage{fixltx2e}
\usepackage{amssymb,amsmath,amsfonts,amsthm}
\usepackage[utf8]{inputenc}
\usepackage{cmap}
\usepackage[T1]{fontenc}
\usepackage[all]{xy}
\usepackage{mathbbol}
\usepackage{graphicx}
\usepackage{xcolor}
\usepackage{lineno}

\newcommand{\black}{\color{black}}

\newcommand{\bl}{\color{blue}{[}\black }
\newcommand{\br}{\color{blue}{]}\black }
\newcommand{\rl}{\color{red}{[}\black }
\newcommand{\rr}{\color{red}{]}\black }

\DeclareMathOperator{\End}{End}

\DeclareMathOperator{\Span}{Span}
\DeclareMathOperator{\Leib}{Leib}

\newtheorem{theorem}[subsection]{Theorem}
\newtheorem{proposition}[subsection]{Proposition}

\theoremstyle{definition}
\newtheorem{definition}[subsection]{Definition}

\theoremstyle{remark}

\begin{document}
%\linenumbers

\title[Finite-dimensional Leibniz algebra representations of $\mathfrak{sl}_2$]{Finite-dimensional Leibniz algebra representations of $\mathfrak{sl}_2$}
\author{T. Kurbanbaev}
\address{[Tuuelbay Kurbanbaev] Institute of Mathematics of Uzbek Academy of Sciences, Mirzo Ulugbek 81, 100170 Tashkent, Uzbekistan.}
\email{tuuelbay@mail.ru}
\author{R. Turdibaev}
\address{[Rustam Turdibaev] Inha University in Tashkent, Ziyolilar 9, 100170 Tashkent, Uzbekistan.}
\email{r.turdibaev@inha.uz}

\thanks{The authors were supported by the project \"E$\Phi$A-$\Phi$\TeX-2018-79.}

\begin{abstract}
	All finite-dimensional Leibniz algebra bimodules of a Lie algebra $\mathfrak{sl}_2$ over a field of characteristic zero are described.
\end{abstract}

\maketitle
\section{Introduction}

Finite-dimensional representations of a finite-dimensional semisimple Lie algebra is a well-studied beautiful classical theory. There is a Weyl's theorem on  complete reducibility that claims that any finite-dimensional module over a semisimple Lie algebra is a direct sum of simple modules. A textbook approach starts with the finite-dimensional representations of the simple Lie algebra $\mathfrak{sl}_2$. In this work we find all finite-dimensional representations of $\mathfrak{sl}_2$ in a larger category -- Leibniz algebra representations of $\mathfrak{sl}_2$.

The notion of a Leibniz algebra first appeared under the name of  D-algebra, introduced by A. Bloh in \cite{Bloh} as one of the generalizations of Lie algebras, in which multiplication by an element is a derivation. Later, they were discovered independently by J.-L. Loday \cite{LodayCyclic} and gain popularity under the name of Leibniz algebras. Given a Leibniz algebra $L$ there is a two-sided ideal $\Leib(L)=\Span\{ [x,x] \mid x\in L\}$, associated to it, also known as the Leibniz kernel by some authors. The canonical Lie algebra $L/\Leib(L)$ is called the liezation of $L$. Due to Leibniz kernel, there are no simple non-Lie Leibniz algebras. However, by abuse of standard terminology a simple Leibniz algebra is introduced in \cite{Dzumadil'daev} as an algebra with simple liezation and simple Leibniz kernel. All such algebras are described via irreducible representations of simple Lie algebras.

While originally defined differently (cf. \cite{Bloh3},\cite{LodayCyclic}), the representation of a Leibniz algebra is given in \cite{LP_Universal} as a $\mathbb{K}$-module $M$ with two actions - left and right, satisfying compatibility conditions coming from a so called square-zero construction. It is known that the category of Leibniz representations of a given Leibniz algebra is not semisimple and any non-Lie Leibniz algebra admits a representation, which is neither simple, nor completely reducible \cite[Proposition 1.2]{arxiv}. In \cite{LP_Leib_rep} the indecomposable objects of the category of Leibniz representations of a Lie algebra are studied and for  $\mathfrak{sl}_2$ the indecomposable objects in that category are described as extensions (see Theorem \ref{LP_Thm} below). Our goal in the current work is to build these extensions explicitly. Remarkably, the authors of \cite{LP_Leib_rep}  prove that for $\mathfrak{sl}_n$ ($n\geq 3$) the category of Leibniz representations is of wild type.

This work is a direct continuation of an investigation started in  \cite{arxiv}. If $M$ is an irreducible Leibniz representation of $\mathfrak{sl}_2$, by Weyl's result the left action on $M$ as a Lie algebra representation decomposes into a direct sum of irreducible Lie representations of $\mathfrak{sl}_2$. Hence, the problem of description reduces to the study of the right action. In case the number of such irreducible Lie representations is two, up to a Leibniz algebra representation isomorphism there are exactly two types of irreducible Leibniz representations, whose actions are described in \cite[Theorem 3.1]{arxiv}. In the current work, we establish the description in full generality in Theorem \ref{main_thm}.

All representations and algebras in this work are finite-dimensional over a field of characteristic zero.

%%%%%%%%%%%%%%%%%%%%%%%%%%%%%%%%%%%

\section{Preliminaries}

\begin{definition} 	An algebra $(L,[-,-])$ over a field $\mathbb{K}$ is called a (right) \emph{Leibniz algebra} if for all $x,y,z\in L$ the following identity holds:
 	$$[x,[y,z]]=[[x,y],z] - [[x,z],y].$$
\end{definition}

In case the bracket is skew-symmetric, the identity above, called Leibniz identity transforms into Jacobi identity. One can establish that category of Lie algebras constitute a full subcategory of category of Leibniz algebras.

Next we use definition from \cite{LP_Universal} to define  representation of a Leibniz algebra.
\begin{definition}
	A $\mathbb{K}$-vector space  $M$ with two bilinear maps $\bl-,-\br:L\times M \rightarrow M$ and $\rl-,-\rr:M\times L \rightarrow M$ is called \textit{representation} of Leibniz algebra $L$ if the following holds:
	\begin{align}
	\rl m,[x,y]\rr & =\rl\rl m,x\rr,y\rr-\rl \rl m,y\rr,x\rr,\label{red}\\
	\bl x,\rl m,y\rr\br & =\rl\bl x,m\br,y\rr-\bl [x,y],m\br,\label{bluered}\\
	\bl x,\bl y,m\br\br & =\bl[x,y],m\br-\rl\bl x,m\br,y\rr.\label{blueblue}
	\end{align}
\end{definition}

Note that, these are exactly the conditions for a direct sum of $\mathbb{K}$-vector spaces $L\oplus M$ to be Leibniz algebra, where $L$ and $M$ are contained as subalgebra and abelian ideal, correspondingly. Such construction is called square-zero construction. Adding identities $(\ref{bluered})$ and $(\ref{blueblue})$ we obtain
\begin{equation}
{\color{blue}[} x, {\color{red}[} m,y {\color{red}]} + {\color{blue}[} y,m
{\color{blue}]} {\color{blue}]} = 0\label{blue}
\end{equation}
which is often used instead of identity (\ref{blueblue}).

Given a representation $M$ of a Leibniz algebra $L$, one defines linear maps  $\lambda_x, \rho_x : M \to M$ by $\lambda_x(m)=\bl x,m \br$ and $\rho_x(m)=\rl m,x \rr$ for every $x\in L, \ m\in M$. Defining relations of Leibniz representation yield for all $x,y\in L$ the following:
\begin{align}
\rho_{[x,y]}=\rho_y\circ\rho_x-\rho_x\circ\rho_y, \label{rho} \\
\lambda_{[x,y]}=\rho_y\circ\lambda_x-\lambda_x\circ\rho_y, \label{lamb1}\\
\lambda_x\circ(\rho_y+\lambda_y)=0.  \label{rholambda1}		
\end{align}

A representation of a Leibniz algebra $L$ is called \textit{symmetric} (\textit{anti-symmetric}) if $\rho_x=-\lambda_x$ ( respectively, $\lambda_x=0$) for all $x\in L$. Considering a Lie algebra $\mathfrak{g}$ as a Leibniz algebra, equation (\ref{rho}) (equation (\ref{red}) for module argument) shows that the map $\rho:\mathfrak{g}\to \End(M)$ defined by $\rho(x)=\rho_x$ coincides with Lie algebra representation of the Lie algebra $\mathfrak{g}$. Moreover, it is known from \cite{LP_Universal} that the category of symmetric, as well as, the category of anti-symmetric representations of a given Leibniz algebra $L$ is equivalent to the category of Lie algebra representations of the liezation of $L$.

For the sake of convenience, throughout this work, for a Leibniz algebra $L$ we call  representation $M$ a \textit{bimodule} $M$, and a Lie algebra module $N$ an $L$-module $N$ or simply a \textit{module} $N$. Given a module $M$ over a Lie algebra $\mathfrak{g}$, one can introduce symmetric and antisymmetric Leibniz bimodules $M^s$ and $M^a$, by taking the left action to be negative of the right action for the first, and identically zero for the second bimodule, correspondingly.

A bimodule is called \textit{simple} or \textit{irreducible}, if it does not admit non-trivial subbimodules. It is well-known that the simple objects in the category of Leibniz representations of a given Leibniz algebra are exactly symmetric and anti-symmetric representations \cite[Lemma 1.9]{Barnes}.

 A bimodule is called \textit{indecomposable}, if it is not a direct sum of its subbimodules. Obviously, a simple bimodule is indecomposable, while the converse is not necessarily true (see \cite[Proposition 1.2]{arxiv} and a paragraph that follows). To study bimodules it suffices to study indecomposable ones. In case Leibniz algebra is a Lie algebra, we utilize the following Weyl's result on complete reducibility of the right action of the bimodule.

\begin{theorem}(\cite{Jacobson})\label{completelyreducible} If $\mathfrak{g}$ is a finite-dimensional semi-simple Lie algebra over a field of characteristic zero, then every finite-dimensional module over $\mathfrak{g}$ is completely reducible.
\end{theorem}

In order to describe all finite-dimensional indecomposable Leibniz bimodules of a Lie algebra $\mathfrak{sl}_2$ over a field of characteristic zero with basis $\{e,f,h\}$ and the products
\begin{center}
	$\begin{array}{lll}
	[e,f]=h, & [e,h]=2e, & [f,h]=-2f,
	\end{array}$
\end{center}
we use the following well-known description of simple $\mathfrak{sl}_2$-modules.
\begin{theorem}\label{irreducible}(\cite{Jacobson}) For every non-negative integer $m$ there exists up to an $\mathfrak{sl}_2$-module isomorphism one and only one irreducible $\mathfrak{sl}_2$-module $V(m)$ of dimension $m+1.$ The module $V(m)$ admits a basis $\{v_0,v_1,\dots,v_m\}$ in which the following holds for all $k=0, \dots, m$:
	\begin{center}
		$\begin{array}{l}
		\, \rl h,v_k\rr=(m-2k)v_k, \\
		\, \rl f,v_k\rr=v_{k+1},   \\
		\, \rl e,v_k\rr=-k(m+1-k)v_{k-1}.\\
		\end{array}$
	\end{center}

\end{theorem}

J.-L. Loday and T. Pirashvili described the Gabriel quiver of Leibniz representations of $\mathfrak{sl}_2$ using Clebsch-Gordon formula in \cite{LP_Leib_rep},  and citing results of \cite{Gabriel} and \cite{Martinez} they found all indecomposable objects in the category of Leibniz representations of $\mathfrak{sl}_2$  as extensions of simple objects. For the sake of convenience, we express their result as the following

\begin{theorem}\cite{LP_Leib_rep}\label{LP_Thm}
	For every non-negative integers $n$ and $k\leq\lfloor n/2 \rfloor+1$ there are exactly two indecomposable $\mathfrak{sl}_2$-bimodules $M_1$ and $M_2$ determined uniquely by the following extensions:
$$0\longrightarrow \bigoplus_{0\leq i<\frac{k}{2}} V(n-4i-2)^a \longrightarrow M_1 \longrightarrow \bigoplus_{0\leq i\leq  \frac{k-1}{2}} V(n-4i)^s\longrightarrow 0,$$
$$0\longrightarrow \bigoplus_{0\leq i\leq  \frac{k-1}{2}} V(n-4i)^a\longrightarrow M_2 \longrightarrow \bigoplus_{0\leq i<\frac{k}{2}} V(n-4i-2)^s \longrightarrow 0,$$
where $V(d)^s$ and $V(d)^a$ are irreducible symmetric and antisymmetric Leibniz representations of $\mathfrak{sl}_2$, correspondingly and $V(0)^a=V(0)^s$ is a trivial one-dimensional representation.

\end{theorem}

Our goal is to build these extensions explicitly. Let $M$ be a finite-dimensional Leibniz bimodule of $\mathfrak{sl}_2$. As a right module, by Theorem \ref{completelyreducible} it is completely reduces into a direct sum of simple $\mathfrak{sl}_2$-modules $V_1\oplus\dots \oplus V_k$, the right action on each simple submodule being described by Theorem \ref{irreducible}. Hence, the study is reduced to the left action only. In the case $k=1$ it is  $V(d)^s$ and $V(d)^a$, i.e. simple objects in the category of Leibniz representation of $\mathfrak{sl}_2$. In \cite[Theorem 3.1]{arxiv} the case $k=2$ is exploited:

\begin{theorem} An $\mathfrak{sl}_2$-module $M=V(n)\oplus V(m)$ is indecomposable as a Leibniz $\mathfrak{sl}_2$-bimodule if and only if $m=n-2$. For any integer $n\geq 2$, up to $\mathfrak{sl}_2$-bimodule isomorphism there are exactly two indecomposable bimodules. The non-zero brackets of the left action are either
\begin{align*}
&\bl h,v_i \br =-(n-2i)v_i -2iw_{i-1}  &&\bl h,w_{j}\br=2(m-j+1)v_{j+1}-(m-2j)w_{j}\\
&\bl f,v_i\br=-v_{i+1}+w_i & \text{or}\hspace*{0.8cm} & \bl f,w_{j}\br=v_{j+2}-w_{j+1}\\
&\bl e,v_{i}\br=i(n-i+1)v_{i-1}+i(i-1)w_{i-2}&& \bl e,w_{j}\br=(m-j+1)((m-j+2)v_{j} +iw_{j-1})
\end{align*}
corresponding to two bimodules, where $\{v_0,\dots,v_n\}$ and $\{w_0,\dots, w_{n-2}\}$ are bases of $V(n)$ and $V(n-2)$ of the Theorem \ref{irreducible}.
\end{theorem}
Note that in the first case $M/V(n-2)$ is symmetric and $V(n-2)$ is anti-symmetric, while in the second one $M/V(n)$ is symmetric and $V(n)$ is anti-symmetric bimodules, that is in accordance with Theorem \ref{LP_Thm}.  In the current work, we use results on the left action established in \cite{arxiv} using only equality (\ref{bluered}). Till the rest of the section  let $M$ be an $\mathfrak{sl}_2$-bimodule that decomposes as a right module into the direct sum $M=V(n)\oplus V(m)$. The next statements shed light on the general form of the left action  depending on $n$ and $m$, that satisfies only identity (\ref{bluered}), where $\{v_0,\dots,v_n\}$ and $\{w_0,\dots, w_{m}\}$ are bases of $V(n)$ and $V(m)$ of the Theorem \ref{irreducible}.

\begin{proposition}\label{l=0}\cite[Proposition 2.6]{arxiv} Let $n=m$. Then identity (\ref{bluered}) implies the following:
	$$\begin{array}{lll}
	\bl h,v_i\br =(n-2i)(\psi_{1} v_i+\psi_{2} w_i), & 0\leq i\leq n, & \\ [1mm]
	
	\bl f,v_i\br =\psi_{1} v_{i+1}+\psi_{2} w_{i+1}, & 0\leq i\leq n-1, & \\ [1mm]
	
	\bl e,v_i\br =-i(n-i+1)(\psi_{1} v_{i-1}+\psi_{2} w_{i-1}), & 1\leq i\leq n, \\
	[1mm]
	
	\bl h,w_i\br =(n-2i)(\psi_{3} v_i+\psi_{4} w_i), & 0\leq i\leq n, &\\
	[1mm]
	
	\bl f,w_i\br =\psi_{3} v_{i+1}+\psi_{4} w_{i+1}, & 0\leq i\leq n-1, & \\
	[1mm]
	\bl e,w_i\br =-i(n-i+1)(\psi_{3} v_{i-1}+\psi_{4} w_{i-1}), & 1\leq i\leq n. \\
	\end{array}$$
\end{proposition}

\begin{proposition}\label{l=1}\cite[Proposition 2.5]{arxiv}
Let  $n=m-2$. Then identity (\ref{bluered}) implies the following:
	$$\begin{array}{llll}
	\bl h,v_i\br =(n-2i)\phi^{11}v_i-2i\phi^{12}w_{i-1}, \ 0 \leq i \leq n,  \\ [1mm]
	\bl f,v_i\br =\phi^{11}v_{i+1}+\phi^{12}w_i, \ 0\leq i \leq n-1, \\ [1mm]
	\bl e,v_i\br =-i(n-i+1)\phi^{11}v_{i-1} + i(i-1)\phi^{12}w_{i-2}, \ 1\leq i \leq n, \\  [1mm]
	\bl h,w_i\br =2(m-i+1)\phi^{21}v_{i+1}+(m-2i)\phi^{22}w_{i}, \ 0 \leq i \leq m,  \\ [1mm]
	\bl f,w_i\br =\phi^{21}v_{i+2}+\phi^{22}w_{i+1}, \ 0\leq i \leq m, \\ [1mm]
	\bl e,w_i\br =(m-i+1)((m-i+2)\phi^{21}v_{i} -i\phi^{22}w_{i-1}), \ 0\leq i \leq m. \\
	\end{array}$$
\end{proposition}

\begin{proposition}\label{l_geq_2}
Let $n-m\geq 4$.  Then identity (\ref{bluered}) implies the following:
	$$	\begin{array}{ll}
	\,	 \bl f,v_i\br =\phi^{11} v_{i+1}   & 0\leq i \leq n-1\\
	\,	 \bl f,w_j\br =\phi^{22} w_{j+1}           & 0\leq j \leq m-1\\		
	\,	 \bl e,v_i\br =-i(n-i+1)\phi^{11}v_{i-1} &         0\leq i \leq n    \\
	\,	 \bl e,w_j\br =-j(m-j+1)\phi^{22}w_{j-1} &       0\leq j \leq m             \\
	\,	 \bl h,v_i\br =(n-2i)\phi^{11}v_i & 0\leq i \leq n \\
	\,	 \bl h,w_i\br =(m-2i)\phi^{22}w_i & 0\leq i \leq m
	\end{array}.$$	
\end{proposition}
\begin{proof}
	From \cite[Proposition 2.2, 2.3, 2.4 ]{arxiv} we have the following table of brackets:
	$$	\begin{array}{ll}
	\,	 \bl f,v_i\br =\phi^{11} v_{i+1}   & 0\leq i \leq n-1\\
	\,	 \bl f,w_j\br =\phi^{22} w_{j+1}           & 0\leq j \leq m-1\\		
	\,	 \bl e,v_i\br =\displaystyle\frac{i(n-i+1)}{n}\epsilon^{11}v_{i-1} &         0\leq i \leq n    \\
	\,	 \bl e,w_j\br =\displaystyle\frac{j(m-j+1)}{m}\epsilon^{22}w_{j-1} &       0\leq j \leq m             \\
	\,	 \bl h,v_i\br =(\eta^{11}-2i\phi^{11})v_i & 0\leq i \leq n \\
	\,	 \bl h,w_i\br =(\eta^{22}-2i\phi^{22})w_i & 0\leq i \leq m
	\end{array}.$$	
	Considering identity (\ref{bluered}) for triples $(f,v_0,e)$ and $(f,w_0,e)$ one obtains $\eta^{11}=n\phi^{11}$ and $\eta^{22}=m\phi^{22}$, correspondingly. Analogously, identity (\ref{bluered}) for $(h,v_i,e)$ and $(h,w_i,e)$ implies $\epsilon^{11}=-n\phi^{11}$ and $\epsilon^{22}=-m\phi^{22}$, correspondingly. This completes the proof.
\end{proof}

\section{Main Results}

Throughout this section $M$ is an $\mathfrak{sl}_2$-bimodule that decomposes into the direct sum of simple $\mathfrak{sl}_2$-modules $M=V(n_1)\oplus V(n_2)\oplus \dots \oplus V(n_k)$. Without loss of generality one can assume that $n_1\geq n_2\geq \dots \geq n_k$. By Theorem \ref{irreducible} each simple module $V_p$ ($1\leq p \leq k$) admits basis $\{v_0^p,v_1^p,\dots,v_{n_p}^p\}$ such that for $0\leq i \leq n_p$ the following holds:
$$
\begin{array}{ll}
{\color{red}[}v_i^p, h{\color{red}]} = (n_p-2i)v_i^p,  \\ [1mm]

{\color{red}[}v_i^p,f{\color{red}]} = v_{i+1}^p, \\ [1mm]

{\color{red}[}v_i^p,e{\color{red}]} = -i(n_p+1-i)v_{i-1}^p.
\end{array}
$$

In general $\bl\mathfrak{sl}_2, V_p\br\subseteq M$ and let us set the following for all $1\leq p \leq k$:
$$ \bl h,v_i^p\br =\sum_{q=1}^k \sum_{j=0}^{n_q} \eta_{ij}^{pq} v_j^q, \ \bl f,v_i^p\br =\sum_{q=1}^k \sum_{j=0}^{n_q} \phi_{ij}^{pq} v_j^q, \ \bl e,v_i^p\br =\sum_{q=1}^k \sum_{j=0}^{n_q} \epsilon_{ij}^{pq} v_j^q.
$$

The description of Leibniz bimodules over $\mathfrak{sl}_2$ is reduced to simplify the left action. As the following proposition shows,  most of the coefficients above are annihilated.

\begin{proposition}\label{odin koeff} Set $l_{pq}=\frac12(n_p-n_q)$. Then
	$$ \bl h,v_i^p \br =\displaystyle \sum_{q=1}^k \eta^{pq}_i v^q_{i-l_{pq}}, \
	\bl f,v_i^p\br=\sum_{q=1}^k \phi^{pq}_i v^q_{i+1-l_{pq}}, \
	\bl e,v_i^p\br=\sum_{q=1}^k \epsilon_i^{pq} v^q_{i-1-l_{pq}},$$
	where $\eta_i^{pq}=\phi^{pq}_i=\epsilon^{pq}_i=0$ if  $l_{pq}\notin \mathbb{Z}$.
\end{proposition}
\begin{proof}
	From $\bl h,{\color{red}[}m,h{\color{red}]}\br=\rl\bl h,m\br,h\rr$ we get
	\begin{multline*}
	(n_p-2i)\sum_{q=1}^k \sum_{j=0}^{n_q} \eta_{ij}^{pq} v_j^q=(n_p-2i)\bl h, v_i^p\br=\bl h,
	\rl v_i^p,h \rr \br \\
	=\rl \bl h,v_i^p \br, h \rr= \rl \sum_{q=1}^k \sum_{j=0}^{n_q} \eta_{ij}^{pq} v_j^q ,h \rr=\sum_{q=1}^k \sum_{j=0}^{n_q} (n_q-2j)\eta_{ij}^{pq} v_j^q.
	\end{multline*}
	Thus $\eta_{ij}^{pq}=0$ unless $j= \frac12(n_q-n_p)+i$. Denote by $\eta^{pq}_i:=\eta_{i, i-l_{pq}}^{pq}$. 
	
	 From $[f,{\color{red}[}m,h{\color{red}]}]= {\color{red}[}[f,m],h{\color{red}]}-2[f, m]$, as above we obtain
$$(n_p-2i)\sum_{q=1}^k \sum_{j=0}^{n_q} \phi_{ij}^{pq} v_j^q=(n_p-2i)[f, v_i^p] =[f,{\color{red}[}v_i^p,h{\color{red}]}] =\sum_{q=1}^k \sum_{j=0}^{n_q} (n_q-2j-2)\phi_{ij}^{pq} v_j^q.$$
	Therefore $\phi_{ij}^{pq}=0$ unless $j= \frac12(n_q-n_p)+i+1$. Denote by $\phi_{i}^{pq}:=\phi_{i,i+1-l_{pq}}^{pq}$ 
	
	From $[e,{\color{red}[}m,h{\color{red}]}]= {\color{red}[}[e,m],h{\color{red}]}-2[e, m]$ we get
$$	(n_p-2i)\sum_{q=1}^k \sum_{j=0}^{n_q} \epsilon_{ij}^{pq} v_j^q=\sum_{q=1}^k \sum_{j=0}^{n_q} (n_q-2j-2)\epsilon_{ij}^{pq} v_j^q.$$
	Hence, $\epsilon_{ij}^{pq}=0$ unless $j= \frac12(n_q-n_p)+i-1$ and denote by $\epsilon_{i}^{pq}:=\epsilon_{i,i-1-l_{pq}}^{pq}$.	
\end{proof}

Next proposition is the main tool in partially reducing the general case to the case $k=2$.

\begin{proposition}\label{projection} For any $x\in \mathfrak{sl}_2$ and $1\leq i \leq j \leq k$, the restriction of the left action $\lambda_x$ on $V(n_i)\oplus V(n_j)$ coincides with the left action described in Propositions \ref{l=0}--\ref{l_geq_2}.
	\end{proposition}

\begin{proof} Let $x\in \mathfrak{sl}_2$ and for any $i,j$ from $\{1,\dots, k\}$ let us denote by $\pi_{i,j}$ the linear projection from $M$ to $V(n_i)\oplus V(n_j)$. Consider $m=v_m^1+\dots+v_m^k\in \oplus_{i=1}^k V(n_i)$. Using the fact that $\rho_x(V(n_i))\subseteq V(n_i)$ for all $1\leq i \leq k$ we have  $$\pi_{ij}(\rho_x(m))=\rho_x(v_m^i)+\rho_x(v_m^j)=\rho_x(v_m^i+v_m^j)=\rho_x(\pi_{ij}(m)).$$
This implies that $\pi_{ij}$ and $\rho_x$ commute. Moreover, using equality (\ref{lamb1}) we have $$\pi_{ij}(\lambda_x\circ \rho_y)=\pi_{ij}(\rho_y\circ \lambda_x-\lambda_{[x,y]})=\rho_y\circ(\pi_{ij}\circ\lambda_x)-\pi_{ij}\circ \lambda_{[x,y]}.$$
Denote by $\lambda_x^{ij}:=\pi_{ij}\circ \lambda_x$. Then $\lambda_x^{ij}\rho_y=\rho_y\lambda_{x}^{ij}-\lambda_{[x,y]}^{ij}$ that shows that $\lambda_x^{ij}$ satisfies equation (\ref{bluered}).  However, for fixed $i$ and $j$ linear maps satisfying such condition are studied in Section (2) of \cite{arxiv} and are described in Propositions \ref{l=0}-\ref{l_geq_2}.
\end{proof}

Although it is known from Theorem \ref{LP_Thm} that for a bimodule $M$ to be indecomposable the sequence $n_1\geq n_2\geq \dots \geq n_k$ must decrease by 2, there is a direct proof why $M$ is decomposable if the sequence mentioned is stable.

\begin{proposition}\label{all_equal}
	Let $M=\oplus_{i=1}^k V_i$, where $\dim V_i=n+1$. Then bimodule $M$ is decomposable.
\end{proposition}

\begin{proof} By Proposition \ref{odin koeff} for all $1 \leq i \leq n+1, \ 1\leq p \leq k$ we have
	$$ \bl h,v_i^p \br =\sum_{q=1}^k \eta^{pq}_i v^q_{i}, \
	\bl f,v_i^p\br=\sum_{q=1}^k \phi^{pq}_i v^q_{i+1}, \
	\bl e,v_i^p\br=\sum_{q=1}^k \epsilon_i^{pq} v^q_{i-1}.$$
	Furthermore, by Proposition \ref{projection} for $(1\leq s,j \leq k)$ and Proposition \ref{l=0} we get the following for all $1 \leq i \leq n + 1, \ 1 \leq p \leq k$:
	\begin{equation}\label{obshaya}
	\bl h,v_i^p\br=(n-2i)\displaystyle\sum_{q=1}^k \phi^{pq} v^q_i, \ \ \bl f,v_i^p\br=\displaystyle\sum_{q=1}^k \phi^{pq} v^q_{i+1}, \ \
	\bl e,v_i^p\br=-i(n-i+1)\displaystyle\sum_{q=1}^k \phi^{pq}v^q_{i+1}.
	\end{equation}

	In the matrix form, we can write the first equality of (\ref{obshaya}) as follows:
	$$
	\left[
	\begin{array}{c}
	\bl h, v_i^1 \br\\
	\bl h, v_i^2 \br\\
	\vdots \\
	\bl h, v_i^k \br\\
	\end{array}
	\right]^T = (n-2i) \left( \left[
	\begin{array}{ccccc}
	\phi^{11} & \phi^{12} & \ldots & \phi^{1k} \\
	\phi^{21} & \phi^{22} & \ldots & \phi^{2k} \\
	.     &      .    & \ldots & .       \\
	\phi^{k1} & \phi^{k2} & \ldots & \phi^{kk} \\
	\end{array}
	\right] \cdot
	\left[
	\begin{array}{c}
	v_i^1 \\
	v_i^2 \\
	\vdots \\
	v_i^k \\
	\end{array}
	\right] \right)^T= (n-2i) \cdot [v_i^1 \ v_i^2 \ ... \ v_i^k] \Phi^T,
	$$
	where $\Phi=(\phi^{ij})_{1\leq i,j\leq k}$ is a matrix. 	
	Verifying identity (\ref{blue}) for  $h$ and $v_i^p$,  $(1 \leq p \leq k)$ we 
	$$
	\left[
	\begin{array}{cccc}
	1+\phi^{11} & \phi^{12}   & \ldots  & \phi^{1k} \\
	\phi^{21}   & 1+\phi^{22} & \ldots  & \phi^{2k} \\
	.       &      .      & \ldots  &     .  \\
	\phi^{k1}   & \phi^{k2}   & \ldots  & 1+\phi^{kk} \\
	\end{array}
	\right] \cdot \Phi \cdot
	\left[
	\begin{array}{c}
	v_i^1 \\
	v_i^2 \\
	\vdots \\
	v_i^k \\
	\end{array}
	\right] =
	\left[
	\begin{array}{c}
	0 \\
	0 \\
	\vdots \\
	0 \\
	\end{array}
	\right].
	$$
Hence, $(I+\Phi)\Phi = O$ and therefore, $\Phi$ is diagonalizable.  Let $\vec{x}=\sum\limits_{q=1}^k x_q v_i^q \in M \ (0\leq i \leq n)$ be an eigenvector of $\Phi^T$ with an eigenvalue $\lambda$. Then 
	$$\bl h,\vec{x}\br=\bl h, x_1 v_i^1 + x_2 v_i^2 +...+ x_k v_i^k\br=[ \bl h,v_i^1\br \ \bl h,v_i^2\br \ ...  \ \bl h,v_i^k\br]
	\cdot \left[\begin{array}{ccc}
	x_1 \\
	x_2 \\
	\vdots \\
	x_k
	\end{array}
	\right]=
	$$
	$$=n [v_i^1 \ v_i^2 \ ... \ v_i^k]\cdot \Phi^T \cdot \left[
	\begin{array}{c}
	x_1 \\
	x_2 \\
	\vdots \\
	x_k \\
	\end{array}
	\right]=(n-2i) [v_i^1 \ v_i^2 \ ... \ v_i^k] \lambda \left[
	\begin{array}{c}
	x_1 \\
	x_2 \\
	\vdots \\
	x_k \\
	\end{array}
	\right]=(n-2i)\lambda \vec{x}.
	$$	
	Consequently, this implies that $\bl \mathfrak {sl}_2, V_i\br \subseteq V_i$, which means the module $M$ is decomposable.
\end{proof}

The following statement describes all subbimodules of $M$ when all $n_i$'s are different.

\begin{proposition}\label{subbimodule}
Let $N$ be a subbimodule of $M$ and $n_i\neq n_j$ for all $1\leq i \neq j\leq k$. Then $N$ is expressed as
$N=V_{n_{i_1}}\oplus V_{n_{i_2}}\oplus \dots \oplus V_{n_{i_t}}$ for some $1\leq i_1<i_2<\dots <i_t\leq k$.
\end{proposition}
\begin{proof}
Let $N$ be a subbimodule of $M$ and $u=(\alpha_{1}v_{p_1}^{i_1}+\dots)+(\alpha_{2}v_{p_2}^{i_2}+\dots)+\dots+(\alpha_{t}v_{p_t}^{i_t}+\dots) \in N$ with $\alpha_{1}\alpha_{2}\dots \alpha_{t}\neq 0$.
Acting with $f$ from the right $(n_{i_1}-p_1)$-times on $u$ we obtain
\begin{equation}\label{p1}
\alpha_{1}v_{n_{i_1}}^{i_1}+(\alpha_{2}v_{q_2}^{i_2}+\dots)+\dots+(\alpha_{t}v_{q_t}^{i_t}+\dots)\in N
\end{equation}
If all the brackets vanish, then $v_{n_{i_1}}^{i_1}\in N$ and acting from the right with $e$ consecutively, one has $V_{n_{i_1}}\subseteq N$. Therefore, $u \mod V_{n_{i_1}} \in N$ and recursively, the process continues.

If some of the brackets are non-zero, apply $h$ from the right to expression (\ref{p1}) and add it to expression (\ref{p1}) multiplied by  $n_{i_1}$,  to deduce
$$(\alpha_{2}(n_{i_1}+n_{i_2}-2q_2)v_{q_2}^{i_2}+\dots)+\dots+(\alpha_{t}(n_1+n_{i_t}-2q_t)v_{q_t}^{i_t}+\dots)\in N.
$$
Note that due to $n_1>n_2>\dots >n_k$, none of the first coefficients is equal to zero in the brackets that did not vanish in expression (\ref{p1}). Hence, we reduce the number of components to one less and recursively we obtain $v_t^{i_t} \in N$. Applying $e$ from the right continuously one has $V_{i_t} \subseteq N$. Therefore, $u \mod V_{i_p} \in N$ and applying the arguments recursively from the start we are done.
\end{proof}

Note that if $n_i=n_j$ for some $i$ and $j$, the result of Proposition \ref{subbimodule} is not true (cf. there are two subbimodules constructed in Case 1 of the proof of \cite[Proposition 3.1]{arxiv}).

\begin{theorem}\label{main_thm} Let $M$ be an $\mathfrak{sl}_2$-bimodule and as a right $\mathfrak{sl}_2$-module let it decompose as $M=V(n_1)\oplus V(n_2)\oplus\dots\oplus V(n_k)$, where $V(n_i)$ are simple $\mathfrak{sl}_2$-modules of Theorem \ref{irreducible} with base $\{v_0^i,\dots,v_{n_i}^i\}, \ 1\leq i\leq k$ and $n_1\geq n_2\geq \dots \geq n_k$. Then $M$ is an indecomposable Leibniz $\mathfrak{sl}_2$-bimodule only if $n_i-n_{i+1}=2$ for all $1\leq i\leq k-1$. Moreover, up to $\mathfrak{sl}_2$-bimodule isomorphism there are exactly two indecomposable $\mathfrak{sl}_2$-bimodules. The non-zero brackets of the left action is either
	$$\begin{array}{lll}
	& \bl h,v_i^{2p}\br=2(n-2p-i+3)v_{i+1}^{2p-1}-(n-2p-2i+2)v_{i+1}^{2p}-2iv_{i-1}^{2p+1},  \\ [1mm]
	& \bl f,v_i^{2p}\br=v_{i+2}^{2p-1}-v_{i+1}^{2p}+v_i^{2p+1}, \\ [1mm]
	& \bl e,v_i^{2p}\br=(n-2p-i+3)((n-2p-i+4)v_{i}^{2p-1}+iv_{i-1}^{2p})+i(i-1)v_{i-2}^{2p+1},
	\end{array}$$
for all $0\leq p \leq k/2$ or	$$\begin{array}{lll}
	&\bl h,v_i^1\br=-(n-2i)v_{i}^1-2iv_{i-1}^2,      \\ [1mm]
	&\bl f,v_i^1\br=-v_{i+1}^1+v_{i}^2,        \\ [1mm]
	&\bl e,v_i^1\br =i(n-i+1)v_{i-1}^1+i(i-1)v_{i-2}^2, \\ [2mm]
	& \bl h,v_i^{2p+1}\br=(n-4p-i+1)v_{i+1}^{2p}-(n-4p-2i)v_{i+1}^{2p+1}-2iv_{i-1}^{2p+2},  \\ [1mm]
	& \bl f,v_i^{2p+1}\br=v_{i+2}^{2p}-v_{i+1}^{2p+1}+v_i^{2p+2}, \\ [1mm]
	& \bl e,v_i^{2p+1}\br=(n-4p-i+1)((n-4p-i+2)v_{i}^{2p}+iv_{i-1}^{2p+1}+i(i-1)v_{i-2}^{2p+2},
	\end{array}$$
for all $1\leq p \leq  (k-1)/2$, where $n=n_1$. 	
\end{theorem}

\begin{proof} By Theorem \ref{LP_Thm} it is clear that the sequence $\{n_i \mid 1\leq i \leq k\}$ must decrease by two. Let us denote by $n=n_1$ and for the sake of convenience, denote by $V_i=V(n-2i+2)=\{v_0^i,v_1^i, \dots, v_{n-2i+2}^i\}, \ 1\leq i \leq k$. First we use Proposition \ref{projection} for pair $(j,j+1)$ for all $1\leq j \leq k-1$ and Proposition \ref{l=1}, then we use Proposition \ref{projection} for pairs $(j,s), \ 1\leq j\leq k-2, \  j+2\leq s \leq k$ and Proposition \ref{l_geq_2} to obtain the following: 	
$$
	\begin{array}{ll}
	\bl h,v_i^1\br=(n-2i)\phi_{1,1}v_i^1 - 2i\phi_{1,2}v_{i-1}^2, \quad\quad 0\leq i\leq n, \\ [1mm]
	\bl f,v_i^1\br=\phi_{1,1}v_{i+1}^1 + \phi_{1,2}v_i^2,         \quad\quad 0\leq i\leq n-1, \\ [1mm]
	\bl e,v_i^1\br=-i(n-i+1)\phi_{1,1}v_{i-1}^1 + i(i-1)\phi_{1,2}v_{i-2}^2, \quad\quad 1\leq i\leq n, \\ [3mm]
	2\leq j \leq k-1, \ 0\leq i\leq n-2j+2: \\ [1mm]
	\bl h,v_i^j\br=2(n-2j-i+3)\phi_{j,j-1}v_{i+1}^{j-1} + (n-2j-2i+2)\phi_{j,j}v_i^j-2i\phi_{j,j+1}v_{i-1}^{j+1},  \\ [1mm]
	\bl f,v_i^j\br=\phi_{j,j-1}v_{i+2}^{j-1}+\phi_{j,j}v_{i+1}^j+\phi_{j,j+1}v_{i}^{j+1}, \\ [1mm]
	\bl e,v_i^j\br=(n-2j+3-i)((n-2j+4-i)\phi_{j,j-1}v_i^{j-1}-i\phi_{j,j}v_{i-1}^j)+i(i-1)\phi_{j,j+1}v_{i-2}^{j+1},  \\ [3mm]
	0\leq i\leq n-2k+2: \\ [1mm]
	\bl h,v_i^k\br=2(n-2k+3-i)\phi_{k,k-1}v_{i+1}^{k-1}+(n-2k+2-2i)\phi_{k,k}v_i^k, \\ [1mm]
	\bl f,v_i^k\br=\phi_{k,k-1}v_{i+2}^{k-1} + \phi_{k,k}v_{i+1}^k, \\ [1mm]
	\bl e,v_i^k\br=(n-2k-i+3)((n-2k+4-i)\phi_{k,k-1}v_i^{k-1}-i\phi_{k,k}v_{i-1}^k. \\
	\end{array}
	$$
	
	Consider identity (\ref{blue}) for corresponding triples:
	
	$\bullet$ For $(f,v_0^1,f)$ we have
	\begin{equation}
	\left\{
	\begin{array}{ll}
	1) & (1+\phi_{1,1}+\phi_{2,2})\phi_{1,2}=0, \\ [1mm]
	2) & (1+\phi_{1,1})\phi_{1,1}+\phi_{1,2}\phi_{2,1}=0, \\ [1mm]
	3) & \phi_{1,2}\phi_{2,3}=0.
	\end{array}\right.\label{fv1f}
	\end{equation}
	
	\
	
	$\bullet$ For $(f,v_0^1,h)$ we get
	\begin{equation}
	\left\{
	\begin{array}{ll}
	1) & (1+\phi_{1,1})\phi_{1,2}=0, \\ [1mm]
	2) & (1+\phi_{1,1})\phi_{1,1}=0.
	\end{array}\right.\label{fv1h}
	\end{equation}
	
	\
	
	$\bullet$ For $(f,v_0^j,f), \ 2\leq j\leq k-1$ we obtain
	\begin{equation}
	\left\{
	\begin{array}{ll}
	1) & (1+\phi_{j-1,j-1}+\phi_{j,j})\phi_{j,j-1}=0, \\ [1mm]
	2) & (1+\phi_{j,j})\phi_{j,j}+\phi_{j,j-1}\phi_{j-1,j}+\phi_{j,j+1}\phi_{j+1,j}=0, \\ [1mm]
	3) & (1+\phi_{j,j}+\phi_{j+1,j+1})\phi_{j,j+1}=0.
	\end{array}\right.\label{fvjf}
	\end{equation}
	
	\
	
	$\bullet$ For $(f,v_0^j,h), \ 2\leq j \leq k-1$ we have
	\begin{equation}
	(1+\phi_{j,j})\phi_{j,j+1}=0. \label{fvjh}
	\end{equation}
	
	\

	$\bullet$ For $(f,v_0^j,e), \ 2\leq j \leq k-1$ we get
	\begin{equation}
	\phi_{j,j-1}\phi_{j-1,j-2}=\phi_{j,j-1}\phi_{j-1,j-1}=\phi_{j,j-1}\phi_{j-1,j}=0. \label{fvje}
	\end{equation}

	Suppose $k$ is odd and consider the following cases (the case when $k$ is an even is carried out analogously).
	
	\
	
\textbf{Case 1}. Let $\phi_{1,1}=0$. Then by (\ref{fv1h}) we have $\phi_{1,2}=0$, hence $\bl\mathfrak{sl}_2,V_1\br=0$.
	
	If $\phi_{2,1}=\phi_{2,3}=0$, then bimodule $M$ is decomposable. Thus $\phi_{2,1}\neq$ and $\phi_{2,3}\neq0$. Hence, from 1) of (\ref{fvjf})  and (\ref{fvje}) we get $\phi_{2,2}=-1$ and $\phi_{3,2}=0$, respectively. Since $\phi_{2,3}\neq0$, then from equality 3) of (\ref{fvjf}) we have $\phi_{3,3}=0$, hence from (\ref{fvjh}) we obtain $\phi_{3,4}=0$. Thus $\bl\mathfrak{sl}_2,V_3\br=0$.
	
	Let $\phi_{4,3}\neq0, \ \phi_{4,5}\neq0$, otherwise $M$ is decomposable. Then in (\ref{fvjf}) equalities 1) and 3) imply $\phi_{4,4}=-1$ and $\phi_{5,5}=0$. Hence by (\ref{fvjh}) and (\ref{fvje}) we obtain $\phi_{5,6}=0$ and $\phi_{5,4}=0$, correspondingly. This means that $\bl\mathfrak{sl}_2,V_5\br=0$. Continuing this process we will get the following:
	$$
	\begin{array}{ll}
	\bl f,v_i^1\br=0,\\ [1mm]
	\bl f,v_i^2\br=\phi_{2,1}v_{i+2}^1-v_{i+1}^2+\phi_{2,3}v_i^3, \\ [1mm]
	\bl f,v_i^3\br=0,\\ [1mm]
	\bl f,v_i^4\br=\phi_{4,3}v_{i+2}^3-v_{i+1}^4+\phi_{4,5}v_i^5, \\ [1mm]
	\bl f,v_i^5\br=0,\\ [1mm]
	. \quad . \quad . \quad . \quad  . \quad . \quad . \quad . \quad \\ [1mm]
	\bl f,v_i^{2p}\br=\phi_{2p,2p-1}v_{i+2}^{2p-1}-v_{i+1}^{2p}+\phi_{2p,2p+1}v_i^{2p+1}, \\ [1mm]
	\bl f,v_i^{2p+1}\br=0,\\
	\end{array}
	$$
	where $6\leq p \leq \frac{k-1}{2}$.
	Make a basis change $(v_i^1)'=v_i^1, \ (v_i^2)'=\displaystyle\frac{1}{\phi_{2,1}}v_i^2, \ (v_i^3)'=\displaystyle\frac{\phi_{2,3}}{\phi_{2,1}}v_i^3,$ $(v_i^4)'=\displaystyle\frac{\phi_{2,3}}{\phi_{2,1}\phi_{4,3}}v_i^4, \ (v_i^5)'=\displaystyle\frac{\phi_{2,3}\phi_{4,5}}{\phi_{2,1}\phi_{4,3}}v_i^5, \ (v_i^6)'=\displaystyle\frac{\phi_{2,3}\phi_{4,5}}{\phi_{2,1}\phi_{4,3}\phi_{6,5}}v_i^6, \dots , (v_i^{2p})'=\displaystyle\frac{\phi_{2,3}\phi_{4,5}\dots\phi_{2p-2,2p-1}}{\phi_{2,1}\phi_{4,3}\dots\phi_{2p-2,2p-3}\phi_{2p,2p-1}}v_i^{2p}$, \ $(v_i^{2p+1})'=\displaystyle\frac{\phi_{2,3}\phi_{4,5}\dots\phi_{2p,2p+1}}{\phi_{2,1}\phi_{4,3}\dots\phi_{2p,2p-1}}v_i^{2p+1}$ to obtain the following
	$$
	\begin{array}{ll}
	\bl f,v_i^1\br=0,\\ [1mm]
	\bl f,v_i^2\br=v_{i+2}^1-v_{i+1}^2+v_i^3, \\ [1mm]
	\bl f,v_i^3\br=0,\\ [1mm]
	. \quad . \quad . \quad . \quad  . \quad . \quad . \quad . \quad \\ [1mm]
	\bl f,v_i^{2p}\br=v_{i+2}^{2p-1}-v_{i+1}^{2p}+v_i^{2p+1}, \\ [1mm]
	\bl f,v_i^{2p+1}\br=0.\\
	\end{array}
	$$
	Thus for all $1\leq p \leq \displaystyle\frac{k-1}{2}$ we obtain
	$$
	\begin{array}{ll}
	\bl h,v_i^{2p}\br=2(n-2p-i+3)v_{i+1}^{2p-1}-(n-2p-2i+2)v_{i+1}^{2p}-2iv_{i-1}^{2p+1}, \quad 0\leq i \leq n-4p+2, \\ [1mm]
	\bl f,v_i^{2p}\br=v_{i+2}^{2p-1}-v_{i+1}^{2p}+v_i^{2p+1}, \quad 0\leq i \leq n-4p+2, \\ [1mm]
	\bl e,v_i^{2p}\br=(n-2p-i+3)((n-2p-i+4)v_{i}^{2p-1}+iv_{i-1}^{2p})+i(i-1)v_{i-2}^{2p+1}, \quad 0\leq i \leq n-4p+2, \\ [3mm]
	\bl h,v_i^{2p+1}\br=0, \quad 0\leq i \leq n-4p, \\ [1mm]
	\bl f,v_i^{2p+1}\br=0, \quad 0\leq i \leq n-4p, \\ [1mm]
	\bl e,v_i^{2p+1}\br=0, \quad 0\leq i \leq n-4p.
	\end{array}
	$$
	
	Using Proposition \ref{subbimodule} it is easy to see that $M$ is indecomposable.
	
	\
	
	\textbf{Case 2}. Let $\phi_{1,1}\neq0$. Then by (\ref{fv1h}) we have $\phi_{1,1}=-1$. Hence in (\ref{fvje}) we get $\phi_{2,1}=0$. If $\phi_{1,2}=0$, then bimodule $M$ is decomposable. So we may assume that  $\phi_{1,2}\neq0$. Then by equations 1) and 3) of  (\ref{fv1f}) one has  $\phi_{2,2}=0$ and $\phi_{2,3}=0$. Hence $\bl \mathfrak{sl}_2,V_2\br=0$. Continuing a similar reasoning as in the Case 1, we obtain for all $1\leq p \leq \displaystyle\frac{k-1}{2}$ the following:
	$$
	\begin{array}{llll}
	\bl h,v_i^1\br=-(n-2i)v_{i}^1-2iv_{i-1}^2, \quad 0\leq i\leq n,         \\ [1mm]
	\bl f,v_i^1\br=-v_{i+1}^1+v_{i}^2, \quad 0\leq i\leq n,                 \\ [1mm]
	\bl e,v_i^1\br=i(n-i+1)v_{i-1}^1+i(i-1)v_{i-2}^2, \quad 0\leq i\leq n,  \\ [3mm]
	\bl h,v_i^{2p}\br=0, \quad 0\leq i \leq n-4p+2, \\ [1mm]
	\bl f,v_i^{2p}\br=0, \quad 0\leq i \leq n-4p+2, \\ [1mm]
	\bl e,v_i^{2p}\br=0, \quad 0\leq i \leq n-4p+2, \\ [3mm]
	\bl h,v_i^{2p+1}\br=(n-4p-i+1)v_{i+1}^{2p}-(n-4p-2i)v_{i+1}^{2p+1}-2iv_{i-1}^{2p+2}, \quad 0\leq i\leq n-4p, \\ [1mm]
	\bl f,v_i^{2p+1}\br=v_{i+2}^{2p}-v_{i+1}^{2p+1}+v_i^{2p+2}, \quad 0\leq i\leq n-4p, \\ [1mm]
	\bl e,v_i^{2p+1}\br=(n-4p-i+1)((n-4p-i+2)v_{i}^{2p}+iv_{i-1}^{2p+1}+i(i-1)v_{i-2}^{2p+2}, \quad 0\leq i\leq n-4p, \\ [1mm]
	\end{array}
	$$

Once again, indecomposability is proved using Proposition \ref{subbimodule}. 	
\end{proof}

%%%%%%%%%%%%%%%%%%%%%%%%%%%%%%%%%%%

\end{document}